\documentclass[11pt,oneside]{article}
\usepackage{afterpage}
\usepackage{fancyhdr}
\usepackage{lipsum}
\newcommand\shorttitle{$S$-curvature of homogeneous Finsler spaces}
\newcommand\authors{g. shanker and s. rani}

\usepackage{supertabular, setspace,hyperref}
\usepackage[a4paper, total={6.5in, 9.5in}]{geometry}
\usepackage[utf8]{inputenc}
\usepackage{amsmath}
\usepackage{mathtools}
\usepackage{amsfonts}
\usepackage{amssymb}
\usepackage{amsthm}
\usepackage{caption}
\usepackage{subcaption}
\usepackage{authblk}
\usepackage[sort]{cite}
\newtheorem{theorem}{Theorem}[section]

\newtheorem{example}[theorem]{Example}
\newtheorem{cor}[theorem]{Corollary}

\newtheorem{remark}{\sc Remark}
\newtheorem{lemma}{\sc Lemma}[section]
\newtheorem{corollary}{\sc Corollary}[section]
\newtheorem{definition}{\sc Definition}[section]

\newcommand{\be}{\begin{eqnarray}}
\newcommand{\ee}{\end{eqnarray}}
\newcommand{\Be}{\begin{eqnarray*}}
	\newcommand{\Ee}{\end{eqnarray*}}
\newcommand{\bee}{\begin{equation}}
\newcommand{\eee}{\end{equation}}
\newcommand{\ba}{\begin{array}}
	\newcommand{\ea}{\end{array}}
\newcommand{\bl}{\begin{lemma}}
	\newcommand{\el}{\end{lemma}}
\newcommand{\bd}{\begin{definition}}
	\newcommand{\ed}{\end{definition}}
\newcommand{\bt}{\begin{theorem}}
	\newcommand{\et}{\end{theorem}}
\newcommand{\bp}{\begin{proof}}
	\newcommand{\ep}{\end{proof}}
\newcommand{\bi}{\begin{itemize}}
	\newcommand{\ei}{\end{itemize}}
\newcommand{\br}{\begin{remark}}
	\newcommand{\er}{\end{remark}}
\newcommand{\bc}{\begin{corollary}}
	\newcommand{\ec}{\end{corollary}}
\newcommand{\bex}{\begin{example}}
	\newcommand{\eex}{\end{example}}

\usepackage{chngcntr}

\counterwithin*{equation}{section}
\counterwithin*{equation}{section}
\begin{document}
	\afterpage{\cfoot{\thepage}}
	\clearpage
	\date{}

	\title{\textbf{On $S$-Curvature of Homogeneous Finsler spaces with  $(\alpha, \beta)$-metrics}}
	
	\author[1]{Gauree Shanker}
	\author[1]{Sarita Rani\thanks{corresponding author, Email: saritas.ss92@gmail.com}}
	\affil[1]{\footnotesize Department of Mathematics and Statistics,
		Central University of Punjab, Bathinda, Punjab-151 001, India}
	\maketitle

	\begin{center}
		\textbf{Abstract}
	\end{center}
	\begin{small}
		The study of curvature properties of homogeneous Finsler spaces with $(\alpha, \beta)$-metrics is one of the central problems in Riemann-Finsler geometry.  In the present paper, the existence of invariant vector fields on homogeneous Finsler spaces with square $(\alpha, \beta)$-metric and Randers changed square $(\alpha, \beta)$-metric is proved. Further, an explicit formula for $S$-curvature of these $(\alpha, \beta)$-metrics is established. Finally, using the formula of $S$-curvature, the mean Berwald curvature of afore said  $(\alpha, \beta)$-metric is calculated. 
	\end{small}\\
	\textbf{Mathematics Subject Classification:} 22E60, 53C30, 53C60.\\
	\textbf{Keywords and Phrases:} Homogeneous Finsler space, square metric, Randers change,   invariant vector field, $S$-curvature, mean Berwald curvature.

	\section{Introduction}
	According to S. S. Chern \cite{Chern1996}, Finsler geometry is just Riemannian geometry without quadratic restriction. Finsler geometry is an interesting and active area of research for both pure and applied reasons \cite{Antonelli,S.I.AmaH.Nag,RGarGWil1996,GKro1934}.  In 1972, M. Matsumoto \cite{M.Mat1972}  introduced the concept of $(\alpha, \beta)-$metrics which are the generalizations of Randers metric introduced by G. Randers \cite{Randers}. Z. Shen \cite{ZShen1997} introduced the notion of $S$-curvature, a non-Riemannian quantity, for a comparison theorem in Finsler geometry. 
	It is non-Riemannian  in the sense  that any Riemannian manifold has vanishing $S$-curvature.   
	One special class of Finsler spaces is homogeneous and symmetric Finsler spaces.  It is an active area of research these days. Many authors \cite{SDeng2009, SDenXWan2010, ZHuSDen2012, GK2019_1, GS2018, MXuSDen2015} have worked in this area.  The main aim of this paper is to establish an explicit formula for $S$-curvature  of homogeneous Finsler spaces with Z. Shen's square metric and Randers change of  square metric. 
	The importance of S-curvature in  Riemann-Finsler geometry can be seen in several  papers (e.g., \cite{ZShen2003, ZShen2005}).
	
	The simplest non-Riemannian metrics are the Randers metrics given by $F=\alpha+ \beta$ with $\lVert \beta\rVert_\alpha <1, $ where $\alpha$ is a Riemannian metric and $\beta$ is a 1-form. Besides Randers metrics, other interesting kind of non-Riemannian metrics are square metrics. Berwald's metric, constructed by Berwald \cite{LBer1929} in 1929 as 
	$$ F=\dfrac{\left( \sqrt{\left( 1-\lvert x\rvert^2\right)\lvert y\rvert^2+\langle x,y \rangle^2 }+\langle x,y \rangle \right)^2 }{\left( 1-\lvert x\rvert^2\right)^2 \sqrt{\left( 1-\lvert x\rvert^2\right)\lvert y\rvert^2+\langle x,y \rangle^2 }} $$ 
	is a classical example of square metric. Berwald's metric can be rewritten as follows:
	
	\begin{equation}{\label{s1}}
	F=\dfrac{\left( \alpha+ \beta\right) ^2}{\alpha},
	\end{equation}
	where $$ \alpha= \dfrac{\sqrt{\left( 1-\lvert x\rvert^2\right)\lvert y\rvert^2+\langle x,y \rangle^2 }}{\left( 1-\lvert x\rvert^2\right)^2}, $$
	and 
	$$\beta=\dfrac{\langle x, y \rangle}{\left( 1-\lvert x\rvert^2\right)^2}. $$
	An $(\alpha, \beta)$-metric expressed in the form (\ref{s1}) is called  square metric \cite{ZSheCYu2014}. Just as Randers metrics, square metrics play an important role in Finsler geometry. The importance of square metric can be seen in papers \cite{ZSheCYu2014, J2QXia2016, CYuHZhu2015}). Square metrics can also be expressed in the form \cite{CYuHZhu2015}
	$$F=\dfrac{\left( \sqrt{\left( 1-b^2\right)\alpha^2+\beta^2 }+\beta \right)^2 }{\left( 1-b^2\right)^2 \sqrt{\left( 1-b^2\right)\alpha^2+\beta^2 }}, $$
	where $ b:= \lVert \beta_x\rVert_\alpha $ is the length of $\beta$. \\
	In this case, $F=\alpha \phi(b^2, \dfrac{\beta}{\alpha}),$ where $\phi = \phi (b^2,s) $ is a smooth function, is called general $(\alpha, \beta)$-metric.
	If $\phi= \phi(s)$ is independent of $b^2,$ then $F$ is called an $(\alpha, \beta)$-metric. One more interesting fact is that if $ \alpha=\lvert y \rvert, $ and 
	$ \beta= \langle x, y \rangle, $ then $F=\lvert y \rvert \phi \left(\lvert x \rvert^2, \dfrac{\langle x, y \rangle}{\lvert y \rvert}\right) $ becomes spherically symmetric metric.\\
	
	\noindent If $F(\alpha, \beta)$ is a Finsler metric, then $ F(\alpha, \beta) \longrightarrow \bar{F}(\alpha, \beta) $ is called a Randers change if 
	\begin{equation}{\label{s2}}
	\bar{F}(\alpha, \beta)=F(\alpha, \beta) + \beta.
	\end{equation}
	Above  change 
	of a Finsler metric has been introduced by M. Matsumoto \cite{M.Mat1974}, and it was named as ``Randers change" by M. Hashiguchi and Y. Ichijy$\bar{o}$ \cite{MHasYIch1980}. 
	In the current paper, we deal with	 following two $(\alpha, \beta)$-metrics:
	\begin{itemize}
		\item [(i)]Square metric $F= \dfrac{(\alpha+\beta)^2}{\alpha}= \alpha \phi(s), \text{where} \ \ \phi(s)=1+s^2+2s.$
		\item [(ii)] Randers changed square metric $F=\dfrac{(\alpha+\beta)^2}{\alpha}+\beta
		= \alpha \phi(s), \text{where} \ \ \phi(s)=1+s^2+3s.$ 
	\end{itemize} 
	
	\noindent The paper is organized as follows:\\
	In section 2, we discuss some basic definitions and results to be used in consequent sections. The existence of invariant vector fields on homogeneous Finsler spaces with afore said metrics is proved in section 3,. Further, in section 4, we establish  an explicit formula for $S$-curvature of afore said $(\alpha, \beta)$-metrics. Finally, in section 5, the mean Berwald curvature of these metrics has been calculated.

	\section{Preliminaries}
	First, we discuss some basic definitions and results required to study afore said spaces.  We refer \cite{BCS, ChernShenRFG} and \cite{SDeng2012} for notations and further details.
	\begin{definition}
		An n-dimensional real vector space $V$ is said to be  a \textbf{Minkowski space}
		if there exists a real valued function $F:V \longrightarrow [0,\infty)$, called Minkowski norm, satisfying the following conditions: 
		\begin{itemize}
			\item  $F$ is smooth on $V \backslash \{0\},$ 
			\item $F$ is positively homogeneous, i.e., $ F(\lambda v)= \lambda F(v), \ \ \forall \ \lambda > 0, $
			\item For any basis $\{u_1,\ u_2, \,..., \ u_n\}$ of $V$ and $y= y^iu_i \in V$, the Hessian matrix \\
			$\left( g_{_{ij}}\right)= \left( \dfrac{1}{2} F^2_{y^i y^j} \right)  $ is positive-definite at every point of $V \backslash \{0\}.$
		\end{itemize}
	\end{definition}
	
	
	\begin{definition}	
		Let $M$ be a  connected smooth manifold. If there exists a function $F: TM \longrightarrow [0, \infty)$ such that $F$ is smooth on the slit tangent bundle $TM \backslash \{0\}$ and the restriction of $F$ to any $T_xM, \ x \in M$, is a Minkowski norm, then $M$ is called a Finsler space  and  $F$ is called a Finsler metric. 
	\end{definition}
	
	An $(\alpha, \beta)$-metric on a connected smooth manifold $M$ is a Finsler metric $F$ constructed from a Riemannian metric $ \alpha = \sqrt{a_{ij}(x) y^i y^j} $ and a one-form $ \beta = b_i(x) y^i $ on $M$ and is of the form  $ F= \alpha \phi \left( \dfrac{\beta}{\alpha}\right),$ where $ \phi  $ is a smooth function on $M$. Basically,  $(\alpha, \beta)$-metrics are the generalization of Randers metrics.   Many authors \cite{SDenXWan2010, ZHuSDen2012, GK2019_1, GK2019,GSK2019, MXuSDen2015} have worked on $(\alpha, \beta)$-metrics. Let us recall Shen's lemma \cite{ChernShenRFG} which provides necessary and sufficient condition for an $(\alpha, \beta)$-metric to be a Finsler metric.

	\begin{lemma} {\label{Shenlemma}}
		Let $F=\alpha \phi(s), \ s=\beta/ \alpha,$ where $\phi $ is a smooth function on an open interval  $ (-b_0, b_0), \ \alpha$ is a Riemannian metric and $\beta$ is a 1-form with  $\lVert \beta \rVert_{\alpha} < b_0$. Then $F$ is a Finsler metric if and only if the following conditions are satisfied:
		$$ \phi(s) > 0, \ \  \phi(s)-s\phi'(s)+\left( b^2-s^2\right) \phi''(s)>0, \ \  \forall \ \ \lvert s\rvert \leq b < b_0.$$
	\end{lemma}
	
	Before defining  homogeneous Finsler spaces, below we discuss some basic concepts required. 
	
	\begin{definition}
		Let $G$ be a smooth manifold having the structure of an abstract group. $G$ is called a Lie group, if the maps $i: G \longrightarrow G$ and $ \mu: G \times G \longrightarrow G  $ defined as $i(g)=g^{-1}, $ and $ \mu(g,h)=gh $ respectively, are smooth. 
	\end{definition}
	Let $G$ be a Lie group and $M$, a smooth manifold. Then a smooth  map $ f: G \times M \longrightarrow M  $ satisfying
	$$ f(g_2, f(g_1,x)) =f(g_2g_1, x), \ \ \text{for all }\ \ g_1, g_2 \in G, \ \ \text{and }\ \ x \in M$$
	is called a smooth action of $G $ on $M.$

	\begin{definition}
		Let $ M $ be a smooth manifold and  $ G, $ a Lie group. If $ G $ acts smoothly  on $ M $, then $ G $ is called a \textbf{Lie transformation group} of $ M $.
	\end{definition}
	The following theorem  gives us a differentiable structure on the coset space of a Lie group.
	\begin{theorem}{\label{DiffStrct.}}
		Let $G$ be a Lie group and $H$, its  closed subgroup. Then there exists a unique differentiable structure on the left coset space $G/H$ with the induced topology that turns $G/H$ into a smooth manifold such that $G$ is a Lie transformation group of $G/H$.
	\end{theorem}
	\begin{definition}
		Let $ (M, F)$  be a  connected Finsler space and  $ I(M, F)$ the group of isometries of $(M,F)$. If the action of $I(M,F)$ is transitive on $M$, then $(M,F)$  is said to be a \textbf{homogeneous Finsler space}. 
	\end{definition}
	
	
	Let $G$ be a Lie group acting transitively on  a smooth manifold $M$. Then for $a \in M$, the isotropy subgroup $G_a$ of $G$ is a closed subgroup and by theorem  \ref{DiffStrct.}, $G$ is a Lie transformation group of $ G/G_a.$ Further,  $ G/G_a $ is diffeomorphic to $M$.
	
	\begin{theorem}\cite{SDeng2012}
		Let $(M,F)$ be a Finsler space. Then $G=I(M,F)$, the group of isometries  of $M$ is a Lie transformation group of $M$. Let $ a \in M$ and $ I_a(M,F)$ be the isotropy subgroup of $I(M,F)$ at $a$. Then $ I_a(M,F)$ is compact.
	\end{theorem}


	Let $ (M,F)$ be a  homogeneous Finsler space, i.e., $G= I(M,F)$ acts transitively on $M$. For $a \in M$, let $ H=I_a(M,F) $ be a closed isotropy subgroup of $G$ which is compact. Then $H$ is a Lie group itself being a closed subgroup of $G$. Write $M$ as the quotient space $G/H$.
	\begin{definition}\cite{Nomizu}{\label{defNomizu}}
		Let $ \mathfrak g$ and $ \mathfrak h $ be the Lie algebras of the Lie groups $G$ and $H$ respectively. Then the direct sum decomposition of $\mathfrak{g} $ as  $ \mathfrak g = \mathfrak h + \mathfrak k,$ where $\mathfrak k$ is a subspace of $\mathfrak g$   such that $\text{Ad}(h) (\mathfrak k) \subset \mathfrak k \ \ \forall \  h \in H,$ is called a reductive decomposition of $\mathfrak g$, and if such decomposition exists, then $ (G/H, F) $ is called reductive homogeneous space.
	\end{definition}
	Therefore, we can write, any homogeneous Finsler space as a coset space of a connected Lie group with an invariant Finsler metric. Here, the Finsler metric $F$ is viewed as $G$ invariant Finsler metric on $M$.\\
	
	\begin{definition}
		A one-parameter subgroup of a Lie group $G$ is a homomorphism $\psi : \mathbb{R} \longrightarrow G,$ such that $\psi(0)=e$, where $e$ is the identity of $G$.
	\end{definition}
	Recall \cite{SDeng2012} the following result which gives us the  existence of one-parameter subgroup of a Lie group.
	\begin{theorem}
		Let $G$ be a Lie group having Lie algebra $\mathfrak g$. Then for any $Y \in \mathfrak g$,  there exists a unique one-parameter subgroup $\psi$ such that $\dot{\psi} (0)=Y_e$, where $e$ is the identity element of $G$.
	\end{theorem} 
	\begin{definition}
		Let $G$ be a Lie group with identity element $e$ and $\mathfrak g$  its Lie algebra.
		The exponential map $\exp: \mathfrak g \longrightarrow G$ is defined by $$\exp(tY)= \psi(t), \ \ \forall \ t \in \mathbb R, $$
		where $\psi : \mathbb{R} \longrightarrow G $ is unique one-parameter subgroup of $G$ with  $\dot{\psi} (0)=Y_e$.
	\end{definition}
	In case of reductive homogeneous manifold, we can identify the tangent space $ T_H(G/H) $ of $G/H$ at the origin $eH=H$ with $\mathfrak k$ through the map
	$$ Y \longmapsto \dfrac{d}{dt}exp(tX)H \arrowvert_{t=0}, \ \ Y \in \mathfrak{k},$$
	since $M$ is identified with $G/H$ and  Lie algebra of any Lie group $G$ is viewed as $T_eG.$

	\section{Invariant Vector Field}
	For a homogeneous Finsler space with square metric $ F=\dfrac{(\alpha + \beta)^2}{\alpha},  $ in theorem \ref{thm3.1},  we prove  the existence of invariant vector field corresponding to 1-form $\beta.$
	For this, first we prove following lemmas:
	
	\begin{lemma}{\label{a1}}
		Let $ (M,\alpha)$ be a Riemannian space and $\beta=b_i y^i$, a   1-form with $\lVert \beta \rVert = \sqrt{b_i b^i} < 1$. Then the square Finsler metric $F=\dfrac{(\alpha+\beta)^2}{\alpha}$,  consists of a Riemannian metric $\alpha$ along with a smooth vector field $X$ on $M$ with $\alpha\left( X\rvert_x\right) < 1$, $ \forall \  x \in M$, 
		i.e., 
		$$ F\left( x, y\right)=  \dfrac{(\alpha\left( x, y\right)+\left\langle  X\rvert_x, y\right\rangle)^2}{\alpha\left( x, y\right) }, \ \ x \in M, \ \ y \in T_x M,$$ 
		where $\left\langle \ , \ \right\rangle $ is the inner product induced by the Riemannian metric $\alpha.$
	\end{lemma}
	\begin{proof}
		We know that  the restriction of a Riemannian metric to a tangent space is an inner product. Therefore, the bilinear form $ \left\langle u, v \right\rangle = a_{ij}u^i v^j, \ \ u, v \ \in  T_x M  $ is an inner product on $T_x M$ for $ x \in M, $ and this inner product induces an inner product on $T^{*}_{x}M,$ the cotangent space of $M$ at $x$ which gives us $  \left\langle dx^i, dx^j \right\rangle = a^{ij}.$ A linear isomorphism exists between $T^{*}_{x}M$ and $T_{x}M,$ which can be defined by using this inner product. It follows that the   1-form $\beta$ corresponds to a smooth vector field $X$ on $M$, which can be written as  $$ X\rvert_x = b^i \dfrac{\partial}{\partial x^i}, \text{\ \ where} \ \   b^i= a^{ij} b_j.$$
		Then, for $ y \in T_xM, $ we have
		$$ \left\langle  X\rvert_x, y\right\rangle = \left\langle b^i \dfrac{\partial}{\partial x^i} \ , \  y^j \dfrac{\partial}{\partial x^j} \right\rangle = b^i y^j a_{ij} = b_j y^j = \beta(y).$$
		Also, we have  
		$$	\alpha^2 (x,y)=a_{ij}y^i y^j,$$
		\text{which implies\ \ }
		$$\alpha^2\left( X\rvert_x\right) =a_{ij}b^i b^j  = \lVert \beta \rVert^2<1,$$
		i. e., 
		$$ \alpha\left( X\rvert_x\right)   <1. $$ This completes the proof.
	\end{proof}
	
	\begin{lemma}{\label{a2}}
		Let $ (M,F)$ be a Finsler space with square metric  $F=\dfrac{(\alpha+\beta)^2}{\alpha}.$ Let  $ I(M, F)$  be the group of isometries  of $ (M,F)$ and $ I(M, \alpha)$ be that of  Riemannian  space $ (M,\alpha).$ Then   $ I(M, F)$   is a closed subgroup of  $ I(M, \alpha).$  
	\end{lemma}
	\begin{proof}
		Let $ x \in M$ and $\phi: (M,F) \longrightarrow (M,F)$ be an isometry. Therefore,   we have 
		$$ F(x,y) = F( \phi(x), d \phi_x(y)), \ \  \forall y\  \in T_x M.$$
		By  Lemma \ref{a1}, we get
		\begin{equation*} 
		\dfrac{(\alpha\left( x, y\right)+\left\langle  X\rvert_x, y\right\rangle)^2} {\alpha\left( x, y\right) } 
		=\dfrac{(\alpha\left( \phi(x), d \phi_x(y)\right)+\left\langle  X\rvert_{\phi(x)}, d \phi_x(y)\right\rangle)^2} {\alpha\left( \phi(x), d \phi_x(y)\right) },
		\end{equation*}
		which gives us 
		\begin{equation}{\label{sq1}}
		\begin{split}
		&	\alpha\left( \phi(x), d \phi_x(y)\right) \alpha^2\left( x, y\right)+ \alpha\left( \phi(x), d \phi_x(y)\right)\left\langle  X\rvert_x, y\right\rangle^2  +2\alpha\left( \phi(x), d \phi_x(y)\right)
		\alpha\left( x, y\right)\left\langle  X\rvert_x, y\right\rangle\\
		&	=\alpha\left( x, y\right) \alpha^2\left( \phi(x), d \phi_x(y)\right) +\alpha\left( x, y\right) \left\langle  X\rvert_{\phi(x)}, d \phi_x(y)\right\rangle^2  +2\alpha\left( x, y\right) \alpha\left( \phi(x), d \phi_x(y)\right)\left\langle  X\rvert_{\phi(x)}, d \phi_x(y)\right\rangle
		\end{split} 
		\end{equation}
		Replacing $y $ by $-y$ in equation (\ref{sq1}), we get
		\begin{equation}{\label{sq2}}
		\begin{split}
		&	\alpha\left( \phi(x), d \phi_x(y)\right) \alpha^2\left( x, y\right)+ \alpha\left( \phi(x), d \phi_x(y)\right)\left\langle  X\rvert_x, y\right\rangle^2  -2\alpha\left( \phi(x), d \phi_x(y)\right)
		\alpha\left( x, y\right)\left\langle  X\rvert_x, y\right\rangle\\
		&	=\alpha\left( x, y\right) \alpha^2\left( \phi(x), d \phi_x(y)\right) +\alpha\left( x, y\right) \left\langle  X\rvert_{\phi(x)}, d \phi_x(y)\right\rangle^2  -2\alpha\left( x, y\right) \alpha\left( \phi(x), d \phi_x(y)\right)\left\langle  X\rvert_{\phi(x)}, d \phi_x(y)\right\rangle
		\end{split}
		\end{equation}
		Subtracting  equation (\ref{sq2}) from equation (\ref{sq1}), we get 
		\begin{equation*}
		\alpha\left( \phi(x), d \phi_x(y)\right)
		\alpha\left( x, y\right)\left\langle  X\rvert_x, y\right\rangle
		=\alpha\left( x, y\right) \alpha\left( \phi(x), d \phi_x(y)\right)\left\langle  X\rvert_{\phi(x)}, d \phi_x(y)\right\rangle,
		\end{equation*}
		which implies
		\begin{equation}{\label{sq3}}
		\left\langle  X\rvert_x, y\right\rangle = \left\langle  X\rvert_{\phi(x)}, d \phi_x(y)\right\rangle. 
		\end{equation}
		Adding  equations (\ref{sq1}) and (\ref{sq2}) and using  equation (\ref{sq3}), we get 
		\begin{equation}
		\begin{split}
		&	\alpha\left( \phi(x), d \phi_x(y)\right) \alpha^2\left( x, y\right)+ \alpha\left( \phi(x), d \phi_x(y)\right)\left\langle  X\rvert_x, y\right\rangle^2  \\
		&	=\alpha\left( x, y\right) \alpha^2\left( \phi(x), d \phi_x(y)\right) +\alpha\left( x, y\right) \left\langle  X\rvert_{x}, y\right\rangle^2, 
		\end{split}
		\end{equation}
		which leads to 
		\begin{equation}
		\alpha\left( x, y\right)= \alpha\left( \phi(x), d \phi_x(y)\right).
		\end{equation}
		Therefore $ \phi $ is an isometry with respect to the Riemannian metric $\alpha$ and $ d \phi_x\left( X\rvert_x\right) = X\rvert_{\phi(x)}.$ Thus $ I(M,F)$ is a closed subgroup of $ I(M,\alpha).$ 
	\end{proof}
	From Lemma (\ref{a2}), we conclude that if $ (M,F)$ is a  homogeneous Finsler space with square metric  $F= \dfrac{(\alpha+\beta)^2}{\alpha}, $
	then the Riemannian space  $(M, \alpha)$ is homogeneous. Further, $M$ can be written as a coset space $G/H$, where $G= I(M,F)$ is a Lie transformation group of $M$ and $H$,  the compact isotropy subgroup $I_a(M,F)$ of $ I(M,F)$ at some point $a \in M$ \cite{DH}. Let $ \mathfrak g$ and $ \mathfrak h $ be the Lie algebras of the Lie groups $G$ and $H$ respectively. If $ \mathfrak g$ can be written as a direct sum of subspaces $\mathfrak h$ and $\mathfrak k$ of $\mathfrak g$   such that $\text{Ad}(h) \mathfrak k \subset \mathfrak k \ \ \forall \  h \in H,$ then from definition \ref{defNomizu}, $(G/H,F)$ is a reductive homogeneous space. \\
	
	Therefore, we can write, any homogeneous Finsler space as a coset space of a connected Lie group with an invariant Finsler metric. Here, the Finsler metric $F$ is viewed as $G$ invariant Finsler metric on $M$.\\
	
	\begin{theorem}{\label{thm3.1}}
		Let   $F= \dfrac{(\alpha+\beta)^2}{\alpha}$ be a $G$-invariant square metric on $G/H,$ $X$ the vector field  corresponding to 1-form $\beta $. Then $ \alpha  $ is a  $G$-invariant Riemannian metric  and the vector field $X$ is  also $G$-invariant. 	
	\end{theorem}
	\begin{proof}
		Since $F$ is a $G$-invariant metric on $G/H$, we have
		$$ F\left( y\right)=F\left( \text{Ad} \left(h \right) y \right), \; \; \forall \;  h \in H, \; \;  y \in \mathfrak k.  $$ 
		By Lemma \ref{a1}, we get
		\begin{equation*} 
		\dfrac{(\alpha\left( y\right)+\left\langle  X, y\right\rangle)^2} {\alpha\left( y\right)} 
		=\dfrac{(\alpha\left( \text{Ad}\left(h \right) y\right)+\left\langle  X, \text{Ad}\left(h \right) y\right\rangle)^2} {\alpha\left( \text{Ad}\left(h \right) y\right) }.
		\end{equation*}
		After simplification, we get
		\begin{equation}{\label{sq5}}
		\begin{split}
		&\alpha\left( \text{Ad}\left(h \right) y\right)\alpha^2\left( y\right)+ \alpha\left( \text{Ad}\left(h \right) y\right)\left\langle  X, y\right\rangle^2
		+2\alpha\left( \text{Ad}\left(h \right) y\right) \alpha\left( y\right)\left\langle  X, y\right\rangle\\
		&	=\alpha\left( y\right) \alpha^2\left( \text{Ad}\left(h \right) y\right) +\alpha\left( y\right)\left\langle  X, \text{Ad}\left(h \right) y\right\rangle^2
		+2\alpha\left( y\right)	\alpha\left( \text{Ad}\left(h \right) y\right)\left\langle  X, \text{Ad}\left(h \right) y\right\rangle.
		\end{split} 
		\end{equation}
		Replacing $y $ by $-y$ in  equation (\ref{sq5}), we get
		\begin{equation}{\label{sq6}}
		\begin{split}
		&\alpha\left( \text{Ad}\left(h \right) y\right)\alpha^2\left( y\right)+ \alpha\left( \text{Ad}\left(h \right) y\right)\left\langle  X, y\right\rangle^2
		-2\alpha\left( \text{Ad}\left(h \right) y\right) \alpha\left( y\right)\left\langle  X, y\right\rangle\\
		&	=\alpha\left( y\right) \alpha^2\left( \text{Ad}\left(h \right) y\right) +\alpha\left( y\right)\left\langle  X, \text{Ad}\left(h \right) y\right\rangle^2
		-2\alpha\left( y\right)	\alpha\left( \text{Ad}\left(h \right) y\right)\left\langle  X, \text{Ad}\left(h \right) y\right\rangle.
		\end{split} 
		\end{equation}
		Subtracting  equation (\ref{sq6}) from equation (\ref{sq5}), we get 
		\begin{equation*}
		\begin{split}
		\alpha\left( \text{Ad}\left(h \right) y\right) \alpha\left( y\right)\left\langle  X, y\right\rangle=\alpha\left( y\right)	\alpha\left( \text{Ad}\left(h \right) y\right)\left\langle  X, \text{Ad}\left(h \right) y\right\rangle,
		\end{split} 
		\end{equation*}
		which gives us 
		\begin{equation}{\label{sq7}}
		\left\langle  X, y\right\rangle=	\left\langle  X, \text{Ad}\left(h \right) y\right\rangle. 
		\end{equation}
		Adding  equations (\ref{sq5}) and (\ref{sq6}) and using  equation (\ref{sq7}), we get 
		\begin{equation*}
		\begin{split}
		\alpha\left( \text{Ad}\left(h \right) y\right)\alpha^2\left( y\right)+ \alpha\left( \text{Ad}\left(h \right) y\right)\left\langle  X, y\right\rangle^2
		=\alpha\left( y\right) \alpha^2\left( \text{Ad}\left(h \right) y\right) +\alpha\left( y\right)\left\langle  X,  y\right\rangle^2
		\end{split} 
		\end{equation*}
		which leads to
		\begin{equation}
		\alpha\left(  y\right)= \alpha\left( \text{Ad}\left(h \right) y\right).
		\end{equation}
		Therefore $ \alpha $ is a $G$-invariant Riemannian metric and $ \text{Ad}\left(h \right) X= X,$ which proves that $X$ is also $G$-invariant. 
	\end{proof}
	\noindent 	Next, for a homogeneous Finsler space with Randers changed square metric $F=\dfrac{(\alpha+\beta)^2}{\alpha}+\beta$, in theorem \ref{thm3.2}, we prove the existence of invariant vector field corresponding to $1$-form $\beta.$  For this, first we prove  following lemmas:

	\begin{lemma}{\label{b1}}
		Let $ (M,\alpha)$ be a Riemannian space and $\beta=b_i y^i$, a   1-form with $\lVert \beta \rVert = \sqrt{b_i b^i} < 1$. Then the Randers changed square Finsler metric $F=\dfrac{(\alpha+\beta)^2}{\alpha}+\beta$,  consists of a Riemannian metric $\alpha$ along with a smooth vector field $X$ on $M$ with $\alpha\left( X\rvert_x\right) < 1$, $ \forall \  x \in M$, 
		i.e., 
		$$ F\left( x, y\right)=  \dfrac{(\alpha\left( x, y\right)+\left\langle  X\rvert_x, y\right\rangle)^2}{\alpha\left( x, y\right) }+\left\langle  X\rvert_x, y\right\rangle, \ \ x \in M, \ \ y \in T_x M,$$ 
		where $\left\langle \ , \ \right\rangle $ is the inner product induced by the Riemannian metric $\alpha.$
	\end{lemma}
	\begin{proof}
		We know that  the restriction of a Riemannian metric to a tangent space is an inner product. Therefore, the bilinear form $ \left\langle u, v \right\rangle = a_{ij}u^i v^j, \ \ u, v \ \in  T_x M  $ is an inner product on $T_x M$ for $ x \in M, $ and this inner product induces an inner product on $T^{*}_{x}M,$ the cotangent space of $M$ at $x$ which gives us $  \left\langle dx^i, dx^j \right\rangle = a^{ij}.$ A linear isomorphism exists between $T^{*}_{x}M$ and $T_{x}M,$ which can be defined by using this inner product. It follows that the   1-form $\beta$ corresponds to a smooth vector field $X$ on $M$, which can be written as  $$ X\rvert_x = b^i \dfrac{\partial}{\partial x^i}, \text{\ \ where} \ \   b^i= a^{ij} b_j.$$
		Then, for $ y \in T_xM, $ we have
		$$ \left\langle  X\rvert_x, y\right\rangle = \left\langle b^i \dfrac{\partial}{\partial x^i} \ , \  y^j \dfrac{\partial}{\partial x^j} \right\rangle = b^i y^j a_{ij} = b_j y^j = \beta(y).$$
		Also, we have  
		$$	\alpha^2 (x,y)=a_{ij}y^i y^j,$$
		\text{which implies\ \ }
		$$\alpha^2\left( X\rvert_x\right) =a_{ij}b^i b^j  = \lVert \beta \rVert^2<1,$$
		i. e., 
		$$ \alpha\left( X\rvert_x\right)   <1. $$ This completes the proof.
	\end{proof}
	
	\begin{lemma}{\label{b2}}
		Let $ (M,F)$ be a Finsler space with Randers changed square Finsler metric  $F=\dfrac{(\alpha+\beta)^2}{\alpha}+\beta$. Let $ I(M, F)$  be the group of isometries  of $ (M,F)$ and $ I(M, \alpha)$ be that of  Riemannian  space $ (M,\alpha).$ Then   $ I(M, F)$   is a closed subgroup of  $ I(M, \alpha).$  
	\end{lemma}
	\begin{proof}
		Let $ x \in M$ and  $\phi: (M,F) \longrightarrow (M,F)$ be an isometry. Therefore,   we have 
		$$ F(x,y) = F( \phi(x), d \phi_x(y)), \ \ \forall y \in T_x M.$$
		By  Lemma \ref{a1}, we get
		\begin{equation*} 
		\dfrac{(\alpha\left( x, y\right)+\left\langle  X\rvert_x, y\right\rangle)^2} {\alpha\left( x, y\right) } +\left\langle  X\rvert_x, y\right\rangle
		=\dfrac{(\alpha\left( \phi(x), d \phi_x(y)\right)+\left\langle  X\rvert_{\phi(x)}, d \phi_x(y)\right\rangle)^2} {\alpha\left( \phi(x), d \phi_x(y)\right) }+\left\langle  X\rvert_{\phi(x)}, d \phi_x(y)\right\rangle,
		\end{equation*}
		which gives us 
		\begin{equation}{\label{sq.1}}
		\begin{split}
		&	\alpha\left( \phi(x), d \phi_x(y)\right) \alpha^2\left( x, y\right)+ \alpha\left( \phi(x), d \phi_x(y)\right)\left\langle  X\rvert_x, y\right\rangle^2  +3\alpha\left( \phi(x), d \phi_x(y)\right)
		\alpha\left( x, y\right)\left\langle  X\rvert_x, y\right\rangle\\
		&	=\alpha\left( x, y\right) \alpha^2\left( \phi(x), d \phi_x(y)\right) +\alpha\left( x, y\right) \left\langle  X\rvert_{\phi(x)}, d \phi_x(y)\right\rangle^2 
		+3\alpha\left( x, y\right) \alpha\left( \phi(x), d \phi_x(y)\right)\left\langle  X\rvert_{\phi(x)}, d \phi_x(y)\right\rangle
		\end{split} 
		\end{equation}
		Replacing $y $ by $-y$ in equation (\ref{sq.1}), we get
		\begin{equation}{\label{sq.2}}
		\begin{split}
		&	\alpha\left( \phi(x), d \phi_x(y)\right) \alpha^2\left( x, y\right)+ \alpha\left( \phi(x), d \phi_x(y)\right)\left\langle  X\rvert_x, y\right\rangle^2  -3\alpha\left( \phi(x), d \phi_x(y)\right)
		\alpha\left( x, y\right)\left\langle  X\rvert_x, y\right\rangle\\
		&	=\alpha\left( x, y\right) \alpha^2\left( \phi(x), d \phi_x(y)\right) +\alpha\left( x, y\right) \left\langle  X\rvert_{\phi(x)}, d \phi_x(y)\right\rangle^2  
		-3\alpha\left( x, y\right) \alpha\left( \phi(x), d \phi_x(y)\right)\left\langle  X\rvert_{\phi(x)}, d \phi_x(y)\right\rangle
		\end{split}
		\end{equation}
		Subtracting  equation (\ref{sq.2}) from equation (\ref{sq.1}), we get 
		\begin{equation*}
		\alpha\left( \phi(x), d \phi_x(y)\right)
		\alpha\left( x, y\right)\left\langle  X\rvert_x, y\right\rangle
		=\alpha\left( x, y\right) \alpha\left( \phi(x), d \phi_x(y)\right)\left\langle  X\rvert_{\phi(x)}, d \phi_x(y)\right\rangle,
		\end{equation*}
		which implies
		\begin{equation}{\label{sq.3}}
		\left\langle  X\rvert_x, y\right\rangle = \left\langle  X\rvert_{\phi(x)}, d \phi_x(y)\right\rangle. 
		\end{equation}
		Adding  equations (\ref{sq.1}) and (\ref{sq.2}) and using  equation (\ref{sq.3}), we get 
		\begin{equation}
		\begin{split}
		&	\alpha\left( \phi(x), d \phi_x(y)\right) \alpha^2\left( x, y\right)+ \alpha\left( \phi(x), d \phi_x(y)\right)\left\langle  X\rvert_x, y\right\rangle^2  \\
		&	=\alpha\left( x, y\right) \alpha^2\left( \phi(x), d \phi_x(y)\right) +\alpha\left( x, y\right) \left\langle  X\rvert_{x}, y\right\rangle^2, 
		\end{split}
		\end{equation}
		which leads to 
		\begin{equation}
		\alpha\left( x, y\right)= \alpha\left( \phi(x), d \phi_x(y)\right).
		\end{equation}
		Therefore $ \phi $ is an isometry with respect to the Riemannian metric $\alpha$ and $ d \phi_x\left( X\rvert_x\right) = X\rvert_{\phi(x)}.$ Thus $ I(M,F)$ is a closed subgroup of $ I(M,\alpha).$ 
	\end{proof}
	
	From Lemma (\ref{b2}), we conclude that if  $ (M,F)$ is a  homogeneous Finsler space with Randers change of square  metric  $F= \dfrac{(\alpha+\beta)^2}{\alpha}+\beta,$ then  the Riemannian space $(M, \alpha)$ is homogeneous. Further,  $M$ can be written as a coset space $G/H$, where $G= I(M,F)$ is a Lie transformation group of $M$ and $H$, 
	the compact isotropy subgroup $I_a(M,F)$ of $ I(M,F)$ at some point $a \in M$ \cite{DH}. Let $ \mathfrak g$ and $ \mathfrak h $ be the Lie algebras of the Lie groups $G$ and $H$ respectively. If $ \mathfrak g$ can be written as a direct sum of subspaces $\mathfrak h$ and $\mathfrak k$ of $\mathfrak g$   such that $\text{Ad}(h) \mathfrak k \subset \mathfrak k \ \ \forall \  h \in H,$ then from definition \ref{defNomizu}, $(G/H,F)$ is a reductive homogeneous space. \\
	
	Therefore, we can write, any homogeneous Finsler space as a coset space of a connected Lie group with an invariant Finsler metric. Here, the Finsler metric $F$ is viewed as $G$ invariant Finsler metric on $M$.\\
	
	\begin{theorem}{\label{thm3.2}}
		Let   $F= \dfrac{(\alpha+\beta)^2}{\alpha}+\beta$ be a $G$-invariant Randers changed square metric on $G/H$. Then $ \alpha  $ is a  $G$-invariant Riemannian metric and the vector field $X$ corresponding to the 1-form $\beta $ is also $G$-invariant. 	
	\end{theorem}
	\begin{proof}
		Since $F$ is a $G$-invariant metric on $G/H$, we have
		$$ F\left( y\right)=F\left( \text{Ad} \left(h \right) y \right), \; \; \forall \;  h \in H, \; \;  y \in \mathfrak k.  $$ 
		By Lemma \ref{b1}, we get
		\begin{equation*} 
		\dfrac{(\alpha\left( y\right)+\left\langle  X, y\right\rangle)^2} {\alpha\left( y\right)} +\left\langle  X, y\right\rangle
		=\dfrac{(\alpha\left( \text{Ad}\left(h \right) y\right)+\left\langle  X, \text{Ad}\left(h \right) y\right\rangle)^2} {\alpha\left( \text{Ad}\left(h \right) y\right) }+\left\langle  X, \text{Ad}\left(h \right) y\right\rangle.
		\end{equation*}
		After simplification, we get
		\begin{equation}{\label{sq.5}}
		\begin{split}
		&\alpha\left( \text{Ad}\left(h \right) y\right)\alpha^2\left( y\right)+ \alpha\left( \text{Ad}\left(h \right) y\right)\left\langle  X, y\right\rangle^2
		+3\alpha\left( \text{Ad}\left(h \right) y\right) \alpha\left( y\right)\left\langle  X, y\right\rangle\\
		&	=\alpha\left( y\right) \alpha^2\left( \text{Ad}\left(h \right) y\right) +\alpha\left( y\right)\left\langle  X, \text{Ad}\left(h \right) y\right\rangle^2
		+3\alpha\left( y\right)	\alpha\left( \text{Ad}\left(h \right) y\right)\left\langle  X, \text{Ad}\left(h \right) y\right\rangle.
		\end{split} 
		\end{equation}
		Replacing $y $ by $-y$ in  equation (\ref{sq.5}), we get
		\begin{equation}{\label{sq.6}}
		\begin{split}
		&\alpha\left( \text{Ad}\left(h \right) y\right)\alpha^2\left( y\right)+ \alpha\left( \text{Ad}\left(h \right) y\right)\left\langle  X, y\right\rangle^2
		-3\alpha\left( \text{Ad}\left(h \right) y\right) \alpha\left( y\right)\left\langle  X, y\right\rangle\\
		&	=\alpha\left( y\right) \alpha^2\left( \text{Ad}\left(h \right) y\right) +\alpha\left( y\right)\left\langle  X, \text{Ad}\left(h \right) y\right\rangle^2
		-3\alpha\left( y\right)	\alpha\left( \text{Ad}\left(h \right) y\right)\left\langle  X, \text{Ad}\left(h \right) y\right\rangle.
		\end{split} 
		\end{equation}
		Subtracting  equation (\ref{sq.6}) from equation (\ref{sq.5}), we get 
		\begin{equation*}
		\begin{split}
		\alpha\left( \text{Ad}\left(h \right) y\right) \alpha\left( y\right)\left\langle  X, y\right\rangle=\alpha\left( y\right)	\alpha\left( \text{Ad}\left(h \right) y\right)\left\langle  X, \text{Ad}\left(h \right) y\right\rangle,
		\end{split} 
		\end{equation*}
		which gives us 
		\begin{equation}{\label{sq.7}}
		\left\langle  X, y\right\rangle=	\left\langle  X, \text{Ad}\left(h \right) y\right\rangle. 
		\end{equation}
		Adding  equations (\ref{sq.5}) and (\ref{sq.6}) and using  equation (\ref{sq.7}), we get 
		\begin{equation*}
		\begin{split}
		\alpha\left( \text{Ad}\left(h \right) y\right)\alpha^2\left( y\right)+ \alpha\left( \text{Ad}\left(h \right) y\right)\left\langle  X, y\right\rangle^2
		=\alpha\left( y\right) \alpha^2\left( \text{Ad}\left(h \right) y\right) +\alpha\left( y\right)\left\langle  X,  y\right\rangle^2
		\end{split} 
		\end{equation*}
		which leads to
		\begin{equation}
		\alpha\left(  y\right)= \alpha\left( \text{Ad}\left(h \right) y\right).
		\end{equation}
		Therefore $ \alpha $ is a $G$-invariant Riemannian metric and $ \text{Ad}\left(h \right) X= X.$ 
	\end{proof}

	The following theorem   gives us  a complete description of invariant vector fields.
	\begin{theorem}\cite{SDengZHou2006}
		There exists a  bijection between the set of invariant vector fields on $G/H $ and the subspace $$ V= \left\lbrace Y  \in \mathfrak k : \text{Ad} \left( h\right) Y= Y, \; \forall \; h  \in H \right\rbrace. $$
	\end{theorem}
	%
	\section{$S$-curvature of homogeneous Finsler space with 
	}

	Now, we discuss  $S$-curvature, a quantity used to measure the rate of change of the volume form of a Finsler space along geodesics. Let $V$ be an n-dimensional real vector space having a basis $ \left\lbrace \alpha_i \right\rbrace $ and $F$ be a Minkowski norm on $V$.   Let $Vol\  B$ to be the volume of a subset $B$  of $\mathbb{R}^n$, and $B^n$ be  the open unit ball. The function $ \tau=\tau(y) $ defined as 
	$$\tau(y)=\ln\left(  \dfrac{\sqrt{det(g_{ij}(y))}}{\sigma_F}\right) , \ \ y \ \in\  V -  \{0\}, $$
	where
	$$ \sigma_F= \dfrac{Vol\left( B^n\right) }{Vol\left\lbrace \left( y^i \right)\in \mathbb{R}^n : F\left( y^i\alpha_i\right) < 1   \right\rbrace },  $$
	is called the distortion of $(V, F).$\\
	For a Finsler space $ (M,F),\  \tau= \tau(x,y)$ is  the distortion of Minkowski norm  $F_x$ on $T_xM, \ x \in M.$ Let $\gamma$ be a geodesic with $\gamma(0)=x, \ \dot{\gamma}(0)=y, $ where $ y \in T_xM, $ then $S$-curvature  denoted as  $S(x,y)$  is the rate of change of distortion along the geodesic $\gamma$, i.e., 
	$$ S(x,y)= \dfrac{d}{dt}\bigg\{\tau\bigg( \gamma(t), \dot{\gamma}(t)\bigg) \bigg\}\bigg|_{t=0}.$$
	Here, it is to be noted that $S(x,y)$ is positively homogeneous of degree one, i.e.,  for $ \lambda > 0, $ we have $ S(x,\lambda y)=\lambda S(x,y).$\\
	
	$S$-curvature of a Finsler space is related to a volume form. There are two important volume forms  in Finsler geometry: the Busemann-Hausdorff volume form $ dV_{BH}=\sigma_{_{BH}}(x)dx$ and the Holmes-Thompson volume form $ dV_{HT}=\sigma_{_{HT}}(x)dx$ defined respectively as 
	
	$$ \sigma_{_{BH}}(x)= \dfrac{Vol\left( B^n\right) }{Vol A },$$
	and 
	$$ \sigma_{_{HT}}(x)= \dfrac{1}{Vol\left( B^n\right)} \int_{A}det\left( g_{ij}\right) dy ,$$\\
	where $ A= \left\lbrace \left( y^i\right) \in \mathbb{R}^n \ \colon\  F\left( x, y^i\frac{\partial}{\partial x^i}\right) < 1 \right\rbrace.$\\
	
	If the Finsler metric $F$ is replaced by a Riemannian metric, then both the volume forms reduce to a single Riemannian  volume form
	$ dV_{HT}=dV_{BH}=\sqrt{det\left( g_{ij}(x) \right) } dx$.\\ 
	
	Next, for the  function $$ T(s)= \phi \left( \phi-s \phi'\right)^{n-2}\left\lbrace  \left( \phi-s \phi'\right) + \left( b^2-s^2\right) \phi'' \right\rbrace, $$
	the volume form $dV= dV_{BH}$ or $dV_{HT}$ is given by $dV= f(b) dV_{\alpha}$, where 
	$$
	f(b)=  \begin{cases}
	\dfrac{\int_{0}^{\pi} \sin^{n-2}t \ dt}{\int_{0}^{\pi}\dfrac{\sin^{n-2} t}{\phi\left( b \cos t\right)^n } \ dt}, & \text{if}\ \  dV= dV_{BH} \\ \\
	\dfrac{ \int_{0}^{\pi}\left( \sin^{n-2} t\right) T\left( b \cos t\right) \ dt }{\int_{0}^{\pi} \sin^{n-2} t \ dt}, & \text{if}\ \  dV= dV_{HT}
	\end{cases}, 
	$$
	and  $dV_{\alpha}= \sqrt{det\left( a_{ij} \right) } dx $  is  the Riemannian volume form of $\alpha.$ \\
	
	The formula for $S$-curvature of an $(\alpha, \beta)$-metric, in local co-ordinate system, introduced by Cheng and Shen \cite{XCheZShe2009}, is as follows:
	\begin{equation}{\label{S-Curv1}}
	S= \left( 2\psi- \dfrac{f'(b)}{bf(b)}\right)\left(r_0+s_o \right) - \dfrac{\Phi}{2 \alpha \Delta^2}\bigg( r_{_{00}}-2\alpha  Qs_{_{0}}\bigg),   
	\end{equation}
	where
	\begin{align*}
	Q&= \dfrac{\phi'}{\phi-s \phi'} \ ,\\
	\Delta &= 1+sQ+\left(b^2-s^2\right)Q'\ ,\\
	\psi&= \dfrac{Q'}{2 \Delta},\\
	\Phi&= \left(sQ'-Q\right) \left( n\Delta + 1 + sQ\right)- \left(b^2-s^2 \right) \left(1+sQ \right)Q'',\\
	r_{ij}&=\dfrac{1}{2} \left(b_{i\mid j}+b_{j \mid i} \right), \ r_j = b^ir_{ij},\  r_0 =r_i y^i,\   r_{00}= r_{ij} y^i y^j,\\
	s_{ij}&= \dfrac{1}{2} \left(b_{i\mid j}-b_{j \mid i} \right), \ s_j= b^i s_{ij} ,
	\  s_0= s_i y^i.
	\end{align*} 
	It is well known \cite{XCheZShe2009} that if the Riemannian length $b$ is constant, then $r_0+s_0=0$. Therefore, in this case, the equation (\ref{S-Curv1}) takes the form
	\begin{equation}{\label{S-Curv2}}
	S= - \dfrac{\Phi}{2 \alpha \Delta^2}\bigg(r_{_{00}}-2 \alpha Qs_{_{0}} \bigg). 
	\end{equation}

	After Shen's work on $S$-curvature, Cheng and Shen \cite{XCheZShe2009} characterized Finsler metrics with isotropic $S$-curvature in 2009. In the same year,  Deng \cite{SDeng2009} gave an explicit formula for $S$-curvature of homogeneous Randers spaces and he proved that a homogeneous Randers space having almost isotropic $S$-curvature has vanishing $S$-curvature. Later in 2010,  Deng and Wang \cite{SDenXWan2010} gave a formula for $S$-curvature of homogeneous $ (\alpha, \beta) $-metrics.  They also derived a formula for mean Berwald curvature $E_{ij}$ of Randers metric. Recently, Shanker and Kaur \cite{GK2019} have proved that there is a mistake in the formula of $S$-curvature given in  \cite{SDenXWan2010}, and they have given the correct version of the formula for $S$-curvature of  homogeneous $ (\alpha, \beta) $-metrics. Further,  some progress has been done in the study of $S$-curvature of homogeneous Finsler spaces (see \cite{ZHuSDen2012}, \cite{MXuSDen2015} for detail).\\
	\begin{definition}{\label{IsoSCurv}}
		Let $(M,F)$ be an $n$-dimensional  Finsler space. If there exists a smooth function $c(x)$ on $M$ and a closed 1-form $\omega$ such that 
		$$ S(x,y)= (n+1) \bigg( c(x) F(y) + \omega(y)\bigg), \ \ x \in M, \   y\in T_x(M), $$
		then  $(M,F)$ is said to have almost isotropic $S$-curvature.
		In addition, if $\omega$ is zero, then   $(M,F)$ is said to have isotropic $S$-curvature.\\ 
		Also, if  $\omega$ is zero and $c(x)$ is constant, then we say, $(M,F)$ has  constant $S$-curvature.
	\end{definition}
	With above notations, let us recall \cite{GK2019} the following theorem:
	
	\begin{theorem}{\label{thmS}}
		
		Let $F= \alpha \phi(s)$   be a $G$-invariant $(\alpha, \beta)$-metric on the reductive  homogeneous Finsler space $G/H$ with a decomposition of the Lie algebra $\mathfrak{g} = \mathfrak{h} +\mathfrak{k}$. Then the $S$-curvature  is given by 
		\begin{equation}{\label{e10}}
		S(H,y)=  \dfrac{\Phi}{2 \alpha \Delta^2}\left(\bigg<\left[ v, y\right] _{\mathfrak k}, y \bigg> + \alpha Q  \bigg<\left[v, y\right] _{\mathfrak k}, v \bigg>\right),
		\end{equation}
		where $v \in \mathfrak{k}$ corresponds to the 1-form $\beta$ and $\mathfrak{k}$ is identified with the tangent space $T_{H}\left( G/H\right) $ of $G/H$ at the origin $H$.
	\end{theorem}
	Now, we establish a formula for $S$-curvature of homogeneous Finsler spaces  with square metric.

	\begin{theorem}{\label{t1}}
		Let $G/H$ be reductive  homogeneous Finsler space  with a decomposition of the Lie algebra $\mathfrak{g} = \mathfrak{h} +\mathfrak{k}$, and  $F= \dfrac{(\alpha+\beta)^2}{\alpha}$   be a $G$-invariant square metric on $G/H$. Then the $S$-curvature  is given by 
		
		\begin{equation}{\label{ee1}}
		S(H,y)= \left[\dfrac{\splitdfrac{-6 ns^3  +2\left(1+n +(2n-1)b^2\right)s}{+3\left(n+1 \right)s^2-\left(1+n\right)(1+2b^2)}}{\left( 1-3s^2+2b^2\right)^2 } \right]\left(\dfrac{2}{1-s} \left\langle \left[v, y\right] _{\mathfrak k}, v \right\rangle+\dfrac{1}{\alpha}\left\langle \left[ v, y\right] _{\mathfrak k}, y \right\rangle\right), 
		\end{equation}
		
		where $v \in \mathfrak{k}$ corresponds to the 1-form $\beta$ and $\mathfrak{k}$ is identified with the tangent space $T_{H}\left( G/H\right) $ of $G/H$ at the origin $H$.
	\end{theorem}
	\begin{proof}
		\noindent For square metric
		$F=  \alpha \phi(s), \text{\ where} \ \ \phi(s)= 1+s^2+2s,$
		the entities  written in the equation (\ref{S-Curv1}) take the values as follows:
		\begin{align*}
		Q=&\dfrac{\phi'}{\phi-s \phi'}= \dfrac{2}{1-s},\\
		Q' =& \dfrac{2}{(1-s)^2},\\
		Q'' =&\dfrac{4}{(1-s)^3},\\
		\Delta=& 1+sQ+\left(b^2-s^2\right)Q'\\
		=&1+ s\left( \dfrac{2}{1-s}\right) + \left( b^2-s^2\right) \dfrac{2}{(1-s)^2}
		= \dfrac{1-3s^2+2b^2}{(1-s)^2}, 
		\end{align*} 
		\begin{align*} 
		\Phi=& \left( sQ'-Q\right) \left(1+ n\Delta + sQ\right)+ \left(s^2-b^2 \right) \left(1+sQ \right)Q''\\
		=&\left( \dfrac{2s}{(1-s)^2}-\dfrac{2}{1-s}\right)\left\{1+ \dfrac{n-3ns^2+2nb^2}{(1-s)^2} +\dfrac{2s}{1-s}\right\}
		+\left(s^2-b^2\right)\left\{ 1 +\dfrac{2s}{1-s}\right\} \left( \dfrac{4}{(1-s)^3}\right) \\
		=&\dfrac{2\biggl\{ -6 ns^3 +3\left(n+1 \right)s^2 +2\left(1+n +(2n-1)b^2\right)s -\left(1+n\right)(1+2b^2)\biggr\}}{(1-s)^4}.
		\end{align*} 
		After substituting these values in equation (\ref{e10}), we get the formula \ref{ee1} for   $S$-curvature of homogeneous Finsler space with square  metric.
	\end{proof}
	\begin{cor}
		Let $G/H$ be reductive  homogeneous Finsler space  with a decomposition of the Lie algebra $\mathfrak{g} = \mathfrak{h} +\mathfrak{k}$, and  $F= \dfrac{(\alpha+\beta)^2}{\alpha}$   be a $G$-invariant square metric on $G/H$.  Then $(G/H, F)$ has isotropic $S$-curvature if and only if  it has vanishing $S$-curvature.
	\end{cor}
	\begin{proof}  Converse part is obvious.
		For  necessary part, suppose $G/H$ has isotropic $S$-curvature, then 
		$$ S(x,y)= (n+1) c(x) F(y) , \ \ x \in G/H, \   y\in T_x(G/H).$$
		Taking $x=H$ and $y = v$ in the equation (\ref{ee1}), we get $c(H)=0.$\\ 
		Consequently $S(H,y)=0 \ \forall \ y \in T_H(G/H). $ \\
		Since $F$ is a homogeneous metric, we have $S=0$ everywhere. \\
	\end{proof}
	Now, we establish a formula for  S-curvature of  homogeneous Finsler space with Randers changes square $(\alpha, \beta)$-metric.\\
	
	\begin{theorem}{\label{t.1}}
		Let $G/H$ be reductive  homogeneous Finsler space  with a decomposition of the Lie algebra $\mathfrak{g} = \mathfrak{h} +\mathfrak{k}$, and  $F= \dfrac{(\alpha+\beta)^2}{\alpha}+\beta$   be a $G$-invariant Randers changed square metric on $G/H$. Then the $S$-curvature  is given by 
		\begin{equation}{\label{e.e1}}
		\begin{split}
		S(H,y)
		=&\left[ \dfrac{\splitdfrac{-12{s}^{5}n+(-27n+9){s}^{4}+ (8 n{b}^{2}+4n -4{b}^{2}  +16){s}^{3}} {+(18n{b}^{2} +18n -18{b}^{2}  +18){s}^{2}+ -12{b}^{2}s-3-6{b}^{2} -6n{b}^{2} -3n}} {2 \left( -3{s}^{2}+1+ 2{b}^{2} \right)  \left( 1 -2{s}^{2} -3{s}^{4}+3s -9{s}^{3} +2{b}^{2} +2{b}^{2}{s}^{2} +6{b}^{2}s\right)}\right] 
		\biggl(\dfrac{2s+3}{1-s^2} \left\langle \left[v, y\right] _{\mathfrak k}, v \right\rangle\\&+\dfrac{1}{\alpha}\left\langle \left[ v, y\right] _{\mathfrak k}, y \right\rangle\biggr),
		\end{split}
		\end{equation}
		where $v \in \mathfrak{k}$ corresponds to the 1-form $\beta$ and $\mathfrak{k}$ is identified with the tangent space $T_{H}\left( G/H\right) $ of $G/H$ at the origin $H$.
	\end{theorem}
	\begin{proof}
		\noindent For Randers changed square $(\alpha, \beta)$-metric 
		$$F
		= \alpha \phi(s), \text{where} \ \ \phi(s)
		=1+s^2+3s,$$
		the entities  written in the equation (\ref{S-Curv1}) take the values as follows:
		\begin{align*}
		Q=&\dfrac{\phi'}{\phi-s \phi'}= \dfrac{2s+3}{1-s^2},\
		Q' = \dfrac{2s^2+6s+2}{(1-s^2)^2},\\
		Q'' =&\dfrac{4s^3+18s^2+12s+6}{(1-s^2)^3},\\
		\Delta=& 1+sQ+\left(b^2-s^2\right)Q'\\
		=&1+ s\left( \dfrac{2s+3}{1-s^2}\right) + \left( b^2-s^2\right) \dfrac{2s^2+6s+2}{(1-s^2)^2}
		= \dfrac{-3s^4-9s^3+(2b^2-2)s^2+(6b^2+3)s+2b^2+1}{(1-s^2)^2},\\
		\Phi=& \left( sQ'-Q\right) \left(1+ n\Delta + sQ\right)+ \left(s^2-b^2 \right) \left(1+sQ \right)Q''\\
		=&\left( \dfrac{2s^3+6s^2+2s}{(1-s^2)^2}-\dfrac{2s+3}{1-s^2}\right)\biggl\{1+ \dfrac{-3ns^4-9ns^3+(2nb^2-2n)s^2+(6nb^2+3n)s+2nb^2+n}{(1-s^2)^2}\\&+\dfrac{2s^2+3s}{1-s^2}\biggr\}
		+\left(s^2-b^2\right)\left\{ 1 +\dfrac{2s^2+3s}{1-s^2}\right\} \left( \dfrac{4s^3+18s^2+12s+6}{(1-s^2)^3}\right) \\
		=&\dfrac{1}{(1-s^2)^4}\biggl\{ -(12n+4)s^7-(63n+21)s^6+(8nb^2-89n-27)s^5+(42nb^2+3n+15)s^4\\&+(62nb^2+58n+40)s^3 +(12nb^2+15n+9)s^2 -(18nb^2+9n+9) s -(6nb^2+3n+3)\biggr\}\\
		&+\dfrac{1}{(1-s^2)^4}\biggl\{ 4s^7+30s^6 +(70-4b^2)s^5+ (60-30b^2)s^4 +(30-70b^2)s^3 +(6-60b^2)s^2 -30b^2 s -6b^2\biggr\}.\\
		=&\dfrac{1}{(1-s^2)^4}\biggl\{ -12ns^7 +(9-63n)s^6 +(8nb^2-4b^2-89n+43)s^5 +(42nb^2-30b^2+3n+75)s^4\\&
		+(62nb^2-70b^2+58n+70)s^3 +(12nb^2-60b^2+15n+15)s^2 -(18nb^2+30b^2+9n+9) s\\& -(6nb^2+6b^2+3n+3)\biggr\}
		\end{align*} 
		After substituting these values in equation (\ref{e10}), we get the formula \ref{e.e1} for   $S$-curvature of homogeneous Finsler space with Randers changed  square metric.
	\end{proof}
	\begin{cor}
		Let $G/H$ be reductive  homogeneous Finsler space  with a decomposition of the Lie algebra $\mathfrak{g} = \mathfrak{h} +\mathfrak{k}$, and  $F= \dfrac{(\alpha+\beta)^2}{\alpha}+\beta$   be a $G$-invariant Randers changed square metric on $G/H$.  Then $(G/H, F)$ has isotropic $S$-curvature if and only if  it has vanishing $S$-curvature.
	\end{cor}
	\begin{proof} Converse part is obvious.
		For  necessary part, suppose $G/H$ has isotropic $S$-curvature, then 
		$$ S(x,y)= (n+1) c(x) F(y) , \ \ x \in G/H, \   y\in T_x(G/H).$$
		Taking $x=H$ and $y = v$ in the equation (\ref{e.e1}), we get $c(H)=0.$\\ 
		Consequently $S(H,y)=0 \ \forall \ y \in T_H(G/H). $ \\
		Since $F$ is a homogeneous metric, we have $S=0$ everywhere. \\
	\end{proof}
	\section{Mean Berwald Curvature}
	\noindent There is another quantity \cite{ChernShenRFG} associated with $S$-curvature called Mean Berwald curvature. \\ Let $ E_{ij}=\dfrac{1}{2} \dfrac{\partial^2}{\partial y^i \partial y^j}S(x,y) = \dfrac{1}{2} \dfrac{\partial^2}{\partial y^i \partial y^j}\left( \dfrac{\partial G^m}{\partial y^m} \right)(x,y),$
	where $G^m$ are spray coefficients.  Then $ \mathcal{E}:=E_{ij} dx^i \otimes dx^j $ is a tensor on $TM\backslash \{0\},$ which we call $E$ tensor. $E$ tensor can also be viewed as a family of symmetric forms $ E_y:T_xM \times T_xM \longrightarrow \mathbb{R} $ defined as
	$$ E_y(u,v)=E_{ij}(x,y) u^i v^j, $$
	where $ u=u^i\dfrac{\partial}{\partial x^i}\arrowvert_x,\ \ v=v^i\dfrac{\partial}{\partial x^i}\arrowvert_x \ \in T_xM. $ Then the collection $ \{ E_y: y \in TM \backslash \{0\} \} $ is called $E$-curvature or mean Berwald curvature. \\
	
	In this section, we calculate  the mean Berwald curvature of the homogeneous Finsler space with afore said metrics.
	To calculate  it, we need the following:\\
	At the origin, $a_{ij}= \delta^i_j,$ \\
	therefore $ y_{_i} =a_{ij}y^j = \delta^i_j y^j =y^i,$
	\begin{align*}
	\alpha_{_{y^i}} &=\dfrac{y_{_i}}{\alpha},\\
	\beta_{_{y^i}} &= b_i,\\
	s_{_{y^i}} &= \dfrac{\partial}{\partial y^i}\left( \dfrac{\beta}{\alpha}\right) = \dfrac{b_i \alpha - s y_{_i}}{\alpha^2},\\
	s_{_{y^i y^j}} &= \dfrac{\partial}{\partial y^j}\left( \dfrac{b_i \alpha - s y_{_i}}{\alpha^2}\right) \\
	&= \dfrac{\alpha^2\left\lbrace b_i \dfrac{y_{_j}}{\alpha}- \left( \dfrac{b_j \alpha -s y_{_j}}{\alpha^2}\right)y_{_i} - s \delta^i_j \right\rbrace - \left( b_i \alpha -s y_{_i}\right)2 \alpha \dfrac{y_{_j}}{\alpha}  }{\alpha^4}\\
	&= \dfrac{-\left( b_i y_{j}+b_j y_{i}\right) \alpha + 3s y_{i} y_j -\alpha^2 s \delta ^i_j }{\alpha ^4},
	\end{align*}
	
	Assuming	
	\begin{align*}
	\dfrac{ -6 ns^3  +3\left(n+1 \right)s^2+\left(2+2n +(4n-2)b^2\right)s -\left(1+n\right)(1+2b^2)}{\left( 1-3s^2+2b^2\right)^2 }=A
	\end{align*} in equation (\ref{ee1}), we find 
	$$ \dfrac{\partial A}{\partial y^j}=
	\dfrac {\splitdfrac{-18n{s}^{4}+ \left( 18n+18 \right) {s}^{3}+ \left(-18{b}^{2}+18 \right) {s}^{2}+ \left(-6n-12n{b}^{2}-6-12{b}^{2} \right)s }{+2+2{b}^{2}-4{b}^{4} +2n+ 8n{b}^{4} +8n{b}^{2}}}{ \left( 1-3\,{s}^{2}+2\,{b}^{2} \right) ^{3}}s_{y^j},$$
	and
	\begin{align*}
	\dfrac{\partial^2 A}{\partial y^i \partial y^j} &=\dfrac{\partial}{\partial y^i}\left[\dfrac {\splitdfrac{-18n{s}^{4}+ \left( 18n+18 \right) {s}^{3}+ \left(-18{b}^{2}+18 \right) {s}^{2}+ \left(-6n-12n{b}^{2}-6-12{b}^{2} \right)s }{+2+2{b}^{2}-4{b}^{4} +2n+ 8n{b}^{4} +8n{b}^{2}}}{ \left( 1-3\,{s}^{2}+2\,{b}^{2} \right) ^{3}}\right] s_{y^j} \\
	&\ \ \ \ \ + \left[\dfrac {\splitdfrac{-18n{s}^{4}+ \left( 18n+18 \right) {s}^{3}+ \left(-18{b}^{2}+18 \right) {s}^{2}+ \left(-6n-12n{b}^{2}-6-12{b}^{2} \right)s }{+2+2{b}^{2}-4{b}^{4} +2n+ 8n{b}^{4} +8n{b}^{2}}}{ \left( 1-3\,{s}^{2}+2\,{b}^{2} \right) ^{3}}\right]s_{y^i y^j}
	\end{align*}
	\begin{align*}
	&=\left[\dfrac {\splitdfrac{\splitdfrac{6\biggl\{-18n{s}^{5}+ \left( 27+27n \right) {s}^{4}+ \left( -24n{b}^{2}-36{b}^{2}+36-12n \right) {s}^{3}} {+\left(-6n-12n{b}^{2}-12{b}^{2}-6 \right){s}^{2}+ \left(12{b}^{2}+6n+24n{b}^{2}-24{b}^{4}+24n{b}^{4}+12 \right) s}}{-1-n-4n{b}^{2}-4{b}^{2}-4n{b}^{4}-4{b}^{4}\biggr\}}}{ \left( 1-3{s}^{2}+2{b}^{2} \right) ^{4}}\right]s_{y^i} s_{y^j} \\
	&+ \left[\dfrac {\splitdfrac{-18n{s}^{4}+ \left( 18n+18 \right) {s}^{3}+ \left(-18{b}^{2}+18 \right) {s}^{2}+ \left(-6n-12n{b}^{2}-6-12{b}^{2} \right)s }{+2+2{b}^{2}-4{b}^{4} +2n+ 8n{b}^{4} +8n{b}^{2}}}{ \left( 1-3\,{s}^{2}+2\,{b}^{2} \right) ^{3}}\right]s_{y^i y^j}.
	\end{align*}
	
	\begin{theorem}
		Let $G/H$ be reductive  homogeneous Finsler space  with a decomposition of the Lie algebra $\mathfrak{g} = \mathfrak{h} +\mathfrak{k}$, and  $F= \dfrac{(\alpha+\beta)^2}{\alpha}$   be a $G$-invariant square metric on $G/H$.  Then the mean Berwald curvature of the homogeneous Finsler space with square metric  is given by 
		\begin{equation}{\label{ECurv_1}}
		\begin{split}
		E_{ij}(H,y)
		=& \dfrac{1}{2}\bigg[ \left( \dfrac{1}{\alpha}\dfrac{\partial^2 A}{\partial y^i \partial y^j}-\dfrac{y_i}{\alpha^3}\dfrac{\partial A}{\partial y^j}-\dfrac{y_j}{\alpha^3}\dfrac{\partial A}{\partial y^i}-\dfrac{A}{\alpha^3}\delta^j_{i}+\dfrac{3A}{\alpha^5}y_{i}y_j\right) \left\langle \left[ v, y\right] _{\mathfrak k}, y \right\rangle  \\
		& + \left( \dfrac{1}{\alpha}\dfrac{\partial A}{\partial y^j} - \dfrac{A } {\alpha^3}y_j\right)\bigg( \left\langle \left[ v, v_i\right] _{\mathfrak k}, y \right\rangle +\left\langle \left[ v, y\right] _{\mathfrak k}, v_i \right\rangle  \bigg) \\
		& +\left( \dfrac{1}{\alpha}\dfrac{\partial A}{\partial y^i}-\dfrac{A}{\alpha^3} y_i\right) \bigg( \left\langle \left[ v, v_j\right] _{\mathfrak k}, y \right\rangle +\left\langle \left[ v, y\right] _{\mathfrak k}, v_j \right\rangle  \bigg) \\
		& + \dfrac{A}{\alpha}\bigg( \left\langle \left[ v, v_j\right] _{\mathfrak k}, v_i \right\rangle +\left\langle \left[ v, v_i\right] _{\mathfrak k}, v_j \right\rangle  \bigg) 
		+
		\biggl\{ \dfrac{2}{1-s} \dfrac{\partial^2 A}{\partial y^i \partial y^j }
		+ \dfrac{2}{(1-s)^2}s_{y^i}\dfrac{\partial A}{\partial y^j} \\&
		+\dfrac{2}{(1-s)^2}s_{y^j}\dfrac{\partial A}{\partial y^i}
		+\dfrac{4A}{(1-s)^3}s_{y^i}s_{y^j}
		+\dfrac{2A}{(1-s)^2}s_{y^i y^j}\biggr\}
		\left\langle \left[v, y\right] _{\mathfrak k}, v\right\rangle \\&
		+\left\{ \dfrac{2}{1-s} \dfrac{\partial A}{\partial y^j} + \dfrac{2A}{(1-s)^2}s_{y^j} \right\} \left\langle \left[v, v_i\right] _{\mathfrak k}, v\right\rangle  
		+ \left\{ \dfrac{2}{1-s} \dfrac{\partial A}{\partial y^i} + \dfrac{2A}{(1-s)^2}s_{y^i}\right\} \left\langle \left[v, v_j\right] _{\mathfrak k}, v\right\rangle
		\bigg],
		\end{split}
		\end{equation}
		where $v \in \mathfrak{k}$ corresponds to the 1-form $\beta$ and $\mathfrak{k}$ is identified with the tangent space $T_{H}\left( G/H\right) $ of $G/H$ at the origin $H$.
	\end{theorem}
	\begin{proof} 
		From equation (\ref{ee1}),  we can write $S$- curvature at the origin as follows:
		%
		$$ S(H,y)= \phi_1+ \psi_1, $$ 
		where $$\phi_1=\dfrac{A}{\alpha} \left\langle \left[ v, y\right] _{\mathfrak k}, y \right\rangle\ \ \text{and} \ \ \psi_1= \dfrac{2A}{1-s}  \left\langle \left[v, y\right] _{\mathfrak k}, v\right\rangle  . $$
		Therefore,  mean Berwald curvature is 
		\begin{equation}{\label{MBC1}}
		E_{ij} = \dfrac{1}{2}\dfrac{\partial^2 S}{\partial y^i \partial y^j}= \dfrac{1}{2}\left( \dfrac{\partial^2 \phi_1}{\partial y^i \partial y^j}+ \dfrac{\partial^2 \psi_1}{\partial y^i \partial y^j}\right),
		\end{equation}
		where\ $ \dfrac{\partial^2 \phi_1}{\partial y^i \partial y^j}$ and $ \dfrac{\partial^2 \psi_1}{\partial y^i \partial y^j} $ are calculated as follows:
		\begin{align*}
		\dfrac{\partial \phi_1}{\partial y^j} &= \dfrac{\partial}{\partial y^j}\left(\dfrac{A}{\alpha} \left\langle \left[ v, y\right] _{\mathfrak k}, y \right\rangle  \right) \\
		&=\left( \dfrac{1}{\alpha}\dfrac{\partial A}{\partial y^j} - \dfrac{A}{\alpha^2} \dfrac{y_j}{\alpha}\right) \left\langle \left[ v, y\right] _{\mathfrak k}, y \right\rangle + \dfrac{A}{\alpha}\bigg( \left\langle \left[ v, v_j\right] _{\mathfrak k}, y \right\rangle + \left\langle \left[ v, y\right] _{\mathfrak k}, v_j \right\rangle  \bigg), 
		\end{align*}
		\begin{align*}
		\dfrac{\partial^2 \phi_1}{\partial y^i \partial y^j} =& \dfrac{\partial}{\partial y^i}\left\lbrace \left( \dfrac{1}{\alpha}\dfrac{\partial A}{\partial y^j} - \dfrac{Ay_j}{\alpha^3} \right) \left\langle \left[ v, y\right] _{\mathfrak k}, y \right\rangle + \dfrac{A}{\alpha}\bigg( \left\langle \left[ v, v_j\right] _{\mathfrak k}, y \right\rangle + \left\langle \left[ v, y\right] _{\mathfrak k}, v_j \right\rangle  \bigg) \right\rbrace \\
		=& \left( \dfrac{1}{\alpha}\dfrac{\partial^2 A}{\partial y^i \partial y^j}-\dfrac{y_i}{\alpha^3}\dfrac{\partial A}{\partial y^j}-\dfrac{y_j}{\alpha^3}\dfrac{\partial A}{\partial y^i}-\dfrac{A}{\alpha^3}\delta^j_{i}+\dfrac{3A}{\alpha^5}y_{i}y_j\right) \left\langle \left[ v, y\right] _{\mathfrak k}, y \right\rangle  \\
		& + \left( \dfrac{1}{\alpha}\dfrac{\partial A}{\partial y^j} - \dfrac{A } {\alpha^3}y_j\right)\bigg( \left\langle \left[ v, v_i\right] _{\mathfrak k}, y \right\rangle +\left\langle \left[ v, y\right] _{\mathfrak k}, v_i \right\rangle  \bigg) \\
		& +\left( \dfrac{1}{\alpha}\dfrac{\partial A}{\partial y^i}-\dfrac{A}{\alpha^3} y_i\right) \bigg( \left\langle \left[ v, v_j\right] _{\mathfrak k}, y \right\rangle +\left\langle \left[ v, y\right] _{\mathfrak k}, v_j \right\rangle  \bigg) 
		+ \dfrac{A}{\alpha}\bigg( \left\langle \left[ v, v_j\right] _{\mathfrak k}, v_i \right\rangle +\left\langle \left[ v, v_i\right] _{\mathfrak k}, v_j \right\rangle  \bigg),  
		\end{align*}
		and 
		\begin{align*}
		\dfrac{\partial \psi_1}{\partial y^j} &=\dfrac{\partial}{\partial y^j}\left(  \dfrac{2A}{1-s} \left\langle \left[v, y\right] _{\mathfrak k}, v\right\rangle \right)\\
		&= \left\{ \dfrac{2}{1-s} \dfrac{\partial A}{\partial y^j} + \dfrac{2A}{(1-s)^2}s_{y^j} \right\} \left\langle \left[v, y\right] _{\mathfrak k}, v\right\rangle 
		+ \dfrac{2A}{1-s} \left\langle \left[v, v_j\right] _{\mathfrak k}, v\right\rangle ,
		\end{align*}
		\begin{align*}
		&	\dfrac{\partial^2 \psi_1}{\partial y^i \partial y^j }
		= \dfrac{\partial}{\partial y^i}\left[\left\{ \dfrac{2}{1-s} \dfrac{\partial A}{\partial y^j} + \dfrac{2A}{(1-s)^2}s_{y^j} \right\} \left\langle \left[v, y\right] _{\mathfrak k}, v\right\rangle 
		+ \dfrac{2A}{1-s} \left\langle \left[v, v_j\right] _{\mathfrak k}, v\right\rangle \right] \\
		&	= \biggl\{ \dfrac{2}{1-s} \dfrac{\partial^2 A}{\partial y^i \partial y^j }
		+ \dfrac{2}{(1-s)^2}s_{y^i}\dfrac{\partial A}{\partial y^j} 
		+\dfrac{2}{(1-s)^2}s_{y^j}\dfrac{\partial A}{\partial y^i}
		+\dfrac{4A}{(1-s)^3}s_{y^i}s_{y^j}
		+\dfrac{2A}{(1-s)^2}s_{y^i y^j}\biggr\}
		\left\langle \left[v, y\right] _{\mathfrak k}, v\right\rangle \\&
		+\left\{ \dfrac{2}{1-s} \dfrac{\partial A}{\partial y^j} + \dfrac{2A}{(1-s)^2}s_{y^j} \right\} \left\langle \left[v, v_i\right] _{\mathfrak k}, v\right\rangle  
		+ \left\{ \dfrac{2}{1-s} \dfrac{\partial A}{\partial y^i} + \dfrac{2A}{(1-s)^2}s_{y^i}\right\} \left\langle \left[v, v_j\right] _{\mathfrak k}, v\right\rangle.
		\end{align*}
		Substituting  all above values in equation (\ref{MBC1}), we get the  formula \ref{ECurv_1}. 
	\end{proof}
	
	Assuming
	$$\dfrac{\splitdfrac{-12{s}^{5}n+(-27n+9){s}^{4}+ (8 n{b}^{2}+4n -4{b}^{2}  +16){s}^{3}} {+(18n{b}^{2} +18n -18{b}^{2}  +18){s}^{2}+ -12{b}^{2}s-3-6{b}^{2} -6n{b}^{2} -3n}} {2 \left( -3{s}^{2}+1+ 2{b}^{2} \right)  \left( 1 -2{s}^{2} -3{s}^{4}+3s -9{s}^{3} +2{b}^{2} +2{b}^{2}{s}^{2} +6{b}^{2}s\right)}=B$$ 
	in equation (\ref{e.e1}), we find 
	\begin{align*}
	\dfrac{\partial B}{\partial y^j}
	=&\dfrac{1}{2} \left(-3 {s}^{2}+1+2 {b}^{2}
	\right)^{-1}  \left( -3 {s}^{4}-9 {s}^{3}+ \left( 2 {b}^{2}-2 \right) {
		s}^{2}+ \left( 3+6 {b}^{2} \right) s+1+2 {b}^{2} \right) ^{-2}
	\biggl\{
	-36 n{s}^{8}\\
	&+ \left( -162 n+54
	\right) {s}^{7}+ \left( -207 n-36 {b}^{2}+225 \right) {s}^{6}+
	\left( -252 {b}^{2}+90 n-36 n{b}^{2}+522 \right) {s}^{5}\\
	&+\left(-488 {b}^{2}+631-80 n{b}^{2}+199 n-8 {b}^{4}+16 n{b}^{4} \right) {
		s}^{4}\\
	&+ \left( -30 n+186+96 n{b}^{4}-408 {b}^{2}-120 n{b}^{2}-48 
	{b}^{4} \right) {s}^{3}\\
	&+ \left( -228 {b}^{2}+156 n{b}^{4}-108 {b}^{
		4}-33-60 n{b}^{2}-69 n \right) {s}^{2}\\
	&+ \left( 96 n{b}^{4}+6 n+6-
	48 {b}^{4}-12 {b}^{2}+60 n{b}^{2} \right) s+9+24 {b}^{2}+36 n{b}^
	{4}+12 {b}^{4}+36 n{b}^{2}+9 n
	\biggr\}
	s_{y^j},
	\end{align*}
	and

	\begin{align*}
	\dfrac{\partial^2 B}{\partial y^i \partial y^j} 
	=&\dfrac{1}{2}\dfrac{\partial}{\partial y^i}\biggl[ \left(-3 {s}^{2}+1+2 {b}^{2}
	\right)^{-1}  \left( -3 {s}^{4}-9 {s}^{3}+ \left( 2 {b}^{2}-2 \right) {
		s}^{2}+ \left( 3+6 {b}^{2} \right) s+1+2 {b}^{2} \right) ^{-2}
	\biggl\{
	-36 n{s}^{8}\\
	&+ \left( -162 n+54
	\right) {s}^{7}+ \left( -207 n-36 {b}^{2}+225 \right) {s}^{6}+
	\left( -252 {b}^{2}+90 n-36 n{b}^{2}+522 \right) {s}^{5}\\
	&+\left(-488 {b}^{2}+631-80 n{b}^{2}+199 n-8 {b}^{4}+16 n{b}^{4} \right) {
		s}^{4}\\
	&+ \left( -30 n+186+96 n{b}^{4}-408 {b}^{2}-120 n{b}^{2}-48 
	{b}^{4} \right) {s}^{3}\\
	&+ \left( -228 {b}^{2}+156 n{b}^{4}-108 {b}^{
		4}-33-60 n{b}^{2}-69 n \right) {s}^{2}\\
	&+ \left( 96 n{b}^{4}+6 n+6-
	48 {b}^{4}-12 {b}^{2}+60 n{b}^{2} \right) s+9+24 {b}^{2}+36 n{b}^
	{4}+12 {b}^{4}+36 n{b}^{2}+9 n
	\biggr\}\biggr] s_{y^j}\\
	+&\dfrac{1}{2} \left(-3 {s}^{2}+1+2 {b}^{2}
	\right)^{-1}  \left( -3 {s}^{4}-9 {s}^{3}+ \left( 2 {b}^{2}-2 \right) {
		s}^{2}+ \left( 3+6 {b}^{2} \right) s+1+2 {b}^{2} \right) ^{-2}
	\biggl\{
	-36 n{s}^{8}\\
	&+ \left( -162 n+54
	\right) {s}^{7}+ \left( -207 n-36 {b}^{2}+225 \right) {s}^{6}+
	\left( -252 {b}^{2}+90 n-36 n{b}^{2}+522 \right) {s}^{5}\\
	&+\left(-488 {b}^{2}+631-80 n{b}^{2}+199 n-8 {b}^{4}+16 n{b}^{4} \right) {
		s}^{4}\\
	&+ \left( -30 n+186+96 n{b}^{4}-408 {b}^{2}-120 n{b}^{2}-48 
	{b}^{4} \right) {s}^{3}\\
	&+ \left( -228 {b}^{2}+156 n{b}^{4}-108 {b}^{
		4}-33-60 n{b}^{2}-69 n \right) {s}^{2}\\
	&+ \left( 96 n{b}^{4}+6 n+6-
	48 {b}^{4}-12 {b}^{2}+60 n{b}^{2} \right) s+9+24 {b}^{2}+36 n{b}^
	{4}+12 {b}^{4}+36 n{b}^{2}+9 n
	\biggr\}s_{y^i y^j}\\
	=&\left( -3 {s}^{2}+1+2 {b}^{2} \right)^{-1}
	\left( -3 {s}
	^{4}-9 {s}^{3}+ \left( 2 {b}^{2}-2 \right) {s}^{2}+ \left( 3+6 {b}^
	{2} \right) s+1+2 {b}^{2} \right) ^{-3}
	\bigg\{
	-108 {s}^{11}n\\
	&+ \left( 243-729 n \right) {s}^{10}+ \left( 
	1593-216 {b}^{2}-144 n{b}^{2}-1935 n \right) {s}^{9}\\&
	+ \left( 5940-
	2052 {b}^{2}-1404 n{b}^{2}-1512 n \right) {s}^{8}\\&
	+ \left( -144 {b}
	^{4}-6570 {b}^{2}+144 n{b}^{4}+1440 n-4338 n{b}^{2}+13356 \right) 
	{s}^{7}\\&
	+ \left( 15894-10188 {b}^{2}+1260 n{b}^{4}+1638 n-1332 {b}^
	{4}-6300 n{b}^{2} \right) {s}^{6}\\&
	+ \left( 8706-4884 {b}^{4}-4254 n{
		b}^{2}-1122 n-8574 {b}^{2}+3756 n{b}^{4} \right) {s}^{5}\\&
	+ \left(3132-7560 {b}^{4}-1080 n+5400 n{b}^{4}-3834 {b}^{2}+54 n{b}^{2}
	\right) {s}^{4}\\&
	+ \left( 2634 n{b}^{2}+3960 n{b}^{4}+1700-4680 {b}^
	{4}+40 {b}^{6}+40 n{b}^{6}+332 n+402 {b}^{2} \right) {s}^{3}\\&
	+\left( 1476 n{b}^{4}+1368 n{b}^{2}+720 {b}^{2}+315 n-1116 {b}^{4
	}+639 \right) {s}^{2}\\&
	+ \left(-90 n{b}^{2}-15 n-162 {b}^{2}-180 n{
		b}^{4}+21-468 {b}^{4}-120 n{b}^{6}-120 {b}^{6} \right) s\\&
	-24-24 n -120 n{b}^{6}-120 {b}^{6}-126 {b}^{2}-216 n{b}^{4}-126 n{b}^{2}-216 {b}^{4}
	\biggr\}s_{y^i}s_{y^j}\\
	\end{align*}
	\begin{align*}
	+&\dfrac{1}{2} \left(-3 {s}^{2}+1+2 {b}^{2}
	\right)^{-1}  \left( -3 {s}^{4}-9 {s}^{3}+ \left( 2 {b}^{2}-2 \right) {
		s}^{2}+ \left( 3+6 {b}^{2} \right) s+1+2 {b}^{2} \right) ^{-2}
	\biggl\{
	-36 n{s}^{8}\\
	&+ \left( -162 n+54
	\right) {s}^{7}+ \left( -207 n-36 {b}^{2}+225 \right) {s}^{6}+
	\left( -252 {b}^{2}+90 n-36 n{b}^{2}+522 \right) {s}^{5}\\
	&+\left(-488 {b}^{2}+631-80 n{b}^{2}+199 n-8 {b}^{4}+16 n{b}^{4} \right) {
		s}^{4}\\
	&+ \left( -30 n+186+96 n{b}^{4}-408 {b}^{2}-120 n{b}^{2}-48 
	{b}^{4} \right) {s}^{3}\\
	&+ \left( -228 {b}^{2}+156 n{b}^{4}-108 {b}^{
		4}-33-60 n{b}^{2}-69 n \right) {s}^{2}\\
	&+ \left( 96 n{b}^{4}+6 n+6-
	48 {b}^{4}-12 {b}^{2}+60 n{b}^{2} \right) s+9+24 {b}^{2}+36 n{b}^
	{4}+12 {b}^{4}+36 n{b}^{2}+9 n
	\biggr\}s_{y^i y^j}
	\end{align*}

	\begin{theorem}
		Let $G/H$ be reductive  homogeneous Finsler space  with a decomposition of the Lie algebra $\mathfrak{g} = \mathfrak{h} +\mathfrak{k}$, and  $F= \dfrac{(\alpha+\beta)^2}{\alpha}+\beta$   be a $G$-invariant Randers changed square metric on $G/H$. Then the mean Berwald curvature of the homogeneous Finsler space with Randers changed  square metric  is given by 
		\begin{equation}{\label{ECurv_2}}
		\begin{split}
		E_{ij}(H,y)
		=& \dfrac{1}{2}\bigg[ \left( \dfrac{1}{\alpha}\dfrac{\partial^2 B}{\partial y^i \partial y^j}-\dfrac{y_i}{\alpha^3}\dfrac{\partial B}{\partial y^j}-\dfrac{y_j}{\alpha^3}\dfrac{\partial B}{\partial y^i}-\dfrac{B}{\alpha^3}\delta^j_{i}+\dfrac{3B}{\alpha^5}y_{i}y_j\right) \left\langle \left[ v, y\right] _{\mathfrak k}, y \right\rangle  \\
		& + \left( \dfrac{1}{\alpha}\dfrac{\partial B}{\partial y^j} - \dfrac{B } {\alpha^3}y_j\right)\bigg( \left\langle \left[ v, v_i\right] _{\mathfrak k}, y \right\rangle +\left\langle \left[ v, y\right] _{\mathfrak k}, v_i \right\rangle  \bigg) \\
		& +\left( \dfrac{1}{\alpha}\dfrac{\partial B}{\partial y^i}-\dfrac{B}{\alpha^3} y_i\right) \bigg( \left\langle \left[ v, v_j\right] _{\mathfrak k}, y \right\rangle +\left\langle \left[ v, y\right] _{\mathfrak k}, v_j \right\rangle  \bigg) \\
		& + \dfrac{B}{\alpha}\bigg( \left\langle \left[ v, v_j\right] _{\mathfrak k}, v_i \right\rangle +\left\langle \left[ v, v_i\right] _{\mathfrak k}, v_j \right\rangle  \bigg) 
		+
		\biggl\{ \dfrac{2s+3}{1-s^2} \dfrac{\partial^2 B}{\partial y^i \partial y^j }
		+ \dfrac{2s^2+6s+2}{(1-s^2)^2}\ s_{y^i}\ \dfrac{\partial B}{\partial y^j} \\&
		+\dfrac{2s^2+6s+2}{(1-s^2)^2}\ s_{y^j}\ \dfrac{\partial B}{\partial y^i}
		+\dfrac{(4s^3+18s^2+12s+6)B}{(1-s^2)^3}s_{y^i}s_{y^j}
		+\dfrac{(2s^2+6s+2)B}{(1-s^2)^2}s_{y^i y^j}\biggr\}
		\left\langle \left[v, y\right] _{\mathfrak k}, v\right\rangle \\&
		+\left\{ \dfrac{2s+3}{1-s^2} \dfrac{\partial B}{\partial y^j} + \dfrac{(2s^2+6s+2)B}{(1-s^2)^2}s_{y^j} \right\} \left\langle \left[v, v_i\right] _{\mathfrak k}, v\right\rangle  \\
		&  + \left\{ \dfrac{2s+3}{1-s^2} \dfrac{\partial B}{\partial y^i} + \dfrac{(2s^2+6s+2)B}{(1-s^2)^2}s_{y^i}\right\} \left\langle \left[v, v_j\right] _{\mathfrak k}, v\right\rangle
		\bigg],
		\end{split}
		\end{equation}
		where $v \in \mathfrak{k}$ corresponds to the 1-form $\beta$ and $\mathfrak{k}$ is identified with the tangent space $T_{H}\left( G/H\right) $ of $G/H$ at the origin $H$.
	\end{theorem}
	\begin{proof}
		From equation (\ref{e.e1}),  we can write $S$- curvature at the origin as follows  
		$$ S(H,y)= \phi_2+ \psi_2, $$ 
		where $$ \phi_2=\dfrac{B}{\alpha} \left\langle \left[ v, y\right] _{\mathfrak k}, y \right\rangle\ \ \text{and} \ \
		\psi_2= \dfrac{2s+3}{1-s^2}\ B\  \left\langle \left[v, y\right] _{\mathfrak k}, v\right\rangle  . $$
		Therefore,  mean Berwald curvature is 
		\begin{equation}{\label{MBC2}}
		E_{ij}= \dfrac{1}{2}\dfrac{\partial^2 S}{\partial y^i \partial y^j}= \dfrac{1}{2}\left( \dfrac{\partial^2 \phi_2}{\partial y^i \partial y^j}+ \dfrac{\partial^2 \psi_2}{\partial y^i \partial y^j}\right),
		\end{equation}
		where\ $ \dfrac{\partial^2 \phi_2}{\partial y^i \partial y^j}$ and $ \dfrac{\partial^2 \psi_2}{\partial y^i \partial y^j} $ are calculated as follows:
		\begin{align*}
		\dfrac{\partial \phi_2}{\partial y^j} &= \dfrac{\partial}{\partial y^j}\left(\dfrac{B}{\alpha} \left\langle \left[ v, y\right] _{\mathfrak k}, y \right\rangle  \right) \\
		&=\left( \dfrac{1}{\alpha}\dfrac{\partial B}{\partial y^j} - \dfrac{B}{\alpha^2} \dfrac{y_j}{\alpha}\right) \left\langle \left[ v, y\right] _{\mathfrak k}, y \right\rangle + \dfrac{B}{\alpha}\bigg( \left\langle \left[ v, v_j\right] _{\mathfrak k}, y \right\rangle + \left\langle \left[ v, y\right] _{\mathfrak k}, v_j \right\rangle  \bigg), 
		\end{align*}
		\begin{align*}
		\dfrac{\partial^2 \phi_2}{\partial y^i \partial y^j} =& \dfrac{\partial}{\partial y^i}\left\lbrace \left( \dfrac{1}{\alpha}\dfrac{\partial B}{\partial y^j} - \dfrac{By_j}{\alpha^3} \right) \left\langle \left[ v, y\right] _{\mathfrak k}, y \right\rangle + \dfrac{B}{\alpha}\bigg( \left\langle \left[ v, v_j\right] _{\mathfrak k}, y \right\rangle + \left\langle \left[ v, y\right] _{\mathfrak k}, v_j \right\rangle  \bigg) \right\rbrace \\
		=& \left( \dfrac{1}{\alpha}\dfrac{\partial^2 B}{\partial y^i \partial y^j}-\dfrac{y_i}{\alpha^3}\dfrac{\partial B}{\partial y^j}-\dfrac{y_j}{\alpha^3}\dfrac{\partial B}{\partial y^i}-\dfrac{B}{\alpha^3}\delta^j_{i}+\dfrac{3B}{\alpha^5}y_{i}y_j\right) \left\langle \left[ v, y\right] _{\mathfrak k}, y \right\rangle  \\
		& + \left( \dfrac{1}{\alpha}\dfrac{\partial B}{\partial y^j} - \dfrac{B } {\alpha^3}y_j\right)\bigg( \left\langle \left[ v, v_i\right] _{\mathfrak k}, y \right\rangle +\left\langle \left[ v, y\right] _{\mathfrak k}, v_i \right\rangle  \bigg) \\
		& +\left( \dfrac{1}{\alpha}\dfrac{\partial B}{\partial y^i}-\dfrac{B}{\alpha^3} y_i\right) \bigg( \left\langle \left[ v, v_j\right] _{\mathfrak k}, y \right\rangle +\left\langle \left[ v, y\right] _{\mathfrak k}, v_j \right\rangle  \bigg)
		+ \dfrac{B}{\alpha}\bigg( \left\langle \left[ v, v_j\right] _{\mathfrak k}, v_i \right\rangle +\left\langle \left[ v, v_i\right] _{\mathfrak k}, v_j \right\rangle  \bigg),  
		\end{align*}
		and 
		\begin{align*}
		\dfrac{\partial \psi_2}{\partial y^j} &=\dfrac{\partial}{\partial y^j}\left(  \dfrac{(2s+3)B}{1-s^2} \left\langle \left[v, y\right] _{\mathfrak k}, v\right\rangle \right)\\
		&= \left\{ \dfrac{2s+3}{1-s^2} \dfrac{\partial B}{\partial y^j} + \dfrac{(2s^2+6s+2)B}{(1-s^2)^2}s_{y^j} \right\} \left\langle \left[v, y\right] _{\mathfrak k}, v\right\rangle 
		+ \dfrac{(2s+3)B}{1-s^2} \left\langle \left[v, v_j\right] _{\mathfrak k}, v\right\rangle ,
		\end{align*}
		\begin{align*}
		\dfrac{\partial^2 \psi_2}{\partial y^i \partial y^j }
		=& \dfrac{\partial}{\partial y^i}\left[\left\{ \dfrac{2s+3}{1-s^2} \dfrac{\partial B}{\partial y^j} 
		+ \dfrac{(2s^2+6s+2)B}{(1-s^2)^2}s_{y^j} \right\} \left\langle \left[v, y\right] _{\mathfrak k}, v\right\rangle 
		+ \dfrac{(2s+3)B}{1-s^2} \left\langle \left[v, v_j\right] _{\mathfrak k}, v\right\rangle \right] \\
		=& \biggl\{ \dfrac{2s+3}{1-s^2} \dfrac{\partial^2 B}{\partial y^i \partial y^j }
		+ \dfrac{2s^2+6s+2}{(1-s^2)^2}\ s_{y^i}\ \dfrac{\partial B}{\partial y^j} 
		+\dfrac{2s^2+6s+2}{(1-s^2)^2}\ s_{y^j}\ \dfrac{\partial B}{\partial y^i}\\&
		+\dfrac{(4s^3+18s^2+12s+6)B}{(1-s^2)^3}s_{y^i}s_{y^j}
		+\dfrac{(2s^2+6s+2)B}{(1-s^2)^2}s_{y^i y^j}\biggr\}
		\left\langle \left[v, y\right] _{\mathfrak k}, v\right\rangle \\&
		+\left\{ \dfrac{2s+3}{1-s^2} \dfrac{\partial B}{\partial y^j} + \dfrac{(2s^2+6s+2)B}{(1-s^2)^2}s_{y^j} \right\} \left\langle \left[v, v_i\right] _{\mathfrak k}, v\right\rangle  \\
		&  + \left\{ \dfrac{2s+3}{1-s^2} \dfrac{\partial B}{\partial y^i} + \dfrac{(2s^2+6s+2)B}{(1-s^2)^2}s_{y^i}\right\} \left\langle \left[v, v_j\right] _{\mathfrak k}, v\right\rangle.
		\end{align*}
		Substituting  all above values in equation (\ref{MBC2}), we get the  formula \ref{ECurv_2}. 
	\end{proof}
	
	
	\section*{Acknowledgments}
	First author is  thankful to Central University of Punjab,  Bathinda  for providing financial assistance as a Research Seed Money grant  via the letter  no. CUPB/CC/17/369.  Second author  is very much thankful to UGC for providing financial assistance in terms of JRF fellowship via letter with  Sr. No. 2061641032 and  Ref. No. 19/06/2016(i)EU-V.

\end{document}